\documentclass{amsart}
\usepackage{amssymb}
\usepackage{graphicx}

\newtheorem{thm}{Theorem}
\newtheorem{cor}{Corollary}
\newtheorem{lem}{Lemma}

\newtheorem{rem}{Remark}
\newtheorem{conj}{Conjecture}

\newcommand{\A}{{\mathcal A}}

\newcommand{\U}{{\mathcal U}}
\newcommand{\es}{{\mathcal S}}

\newcommand{\D}{{\mathbb D}}

\newcommand{\real}{{\operatorname{Re}\,}}

\def\be{\begin{equation}}
\def\ee{\end{equation}}

\newcommand{\bee}{\begin{enumerate}}
\newcommand{\eee}{\end{enumerate}}

\newcommand{\blem}{\begin{lem}}
\newcommand{\elem}{\end{lem}}
\newcommand{\bthm}{\begin{thm}}
\newcommand{\ethm}{\end{thm}}
\newcommand{\bcor}{\begin{cor}}
\newcommand{\ecor}{\end{cor}}
\newcommand{\beg}{\begin{example}}
\newcommand{\eeg}{\end{example}}
\newcommand{\begs}{\begin{examples}}
\newcommand{\eegs}{\end{examples}}
\newcommand{\bdefe}{\begin{defin}}
\newcommand{\edefe}{\end{defin}}
\newcommand{\bprob}{\begin{prob}}
\newcommand{\eprob}{\end{prob}}
\newcommand{\bei}{\begin{itemize}}
\newcommand{\eei}{\end{itemize}}

\newcommand{\bcon}{\begin{conj}}
\newcommand{\econ}{\end{conj}}
\newcommand{\bcons}{\begin{conjs}}
\newcommand{\econs}{\end{conjs}}
\newcommand{\bprop}{\begin{propo}}
\newcommand{\eprop}{\end{propo}}
\newcommand{\br}{\begin{rem}}
\newcommand{\er}{\end{rem}}
\newcommand{\brs}{\begin{rems}}
\newcommand{\ers}{\end{rems}}
\newcommand{\bo}{\begin{obser}}
\newcommand{\eo}{\end{obser}}
\newcommand{\bos}{\begin{obsers}}
\newcommand{\eos}{\end{obsers}}
\newcommand{\bpf}{\begin{pf}}
\newcommand{\epf}{\end{pf}}
\newcommand{\ba}{\begin{array}}
\newcommand{\ea}{\end{array}}
\newcommand{\beq}{\begin{eqnarray}}
\newcommand{\beqq}{\begin{eqnarray*}}
\newcommand{\eeq}{\end{eqnarray}}
\newcommand{\eeqq}{\end{eqnarray*}}

\begin{document}
\bibliographystyle{amsplain}

\title[On certain applications of Grunsky coefficients]{On certain applications of Grunsky coefficients in the theory of univalent functions}

\author[M. Obradovi\'{c}]{Milutin Obradovi\'{c}}
\address{Department of Mathematics,
Faculty of Civil Engineering, University of Belgrade,
Bulevar Kralja Aleksandra 73, 11000, Belgrade, Serbia}
\email{obrad@grf.bg.ac.rs}

\author[N. Tuneski]{Nikola Tuneski}
\address{Department of Mathematics and Informatics, Faculty of Mechanical Engineering, Ss. Cyril and Methodius
University in Skopje, Karpo\v{s} II b.b., 1000 Skopje, Republic of North Macedonia.}
\email{nikola.tuneski@mf.edu.mk}


\dedicatory{This paper is dedicated to the 65$^{th}$ birth anniversary of Prof. Saminathan Ponnusamy.}

\begin{abstract}
In this paper a survey is given of application of a method based on Grunsky coefficients for obtaining different estimates (some sharp) for the general class of univalent functions where no analytical characterisation exists. More precisely, estimates are given for the modulus of the third and the fourth logarithmic coefficients, for the modulus of the second and the third Hankel determinant for the general class of univalent functions, and for the modulus of some coefficients of the inverse function, and some coefficient differences.
\end{abstract}

\maketitle

\section{Introduction and definitions}

\medskip

Let $\mathcal{A}$ be the class of functions $f$ which are analytic  in the open unit disc $\D=\{z:|z|<1\}$ of the form
\be\label{e1}
f(z)=z+a_2z^2+a_3z^3+\cdots,
\ee
and let $\mathcal{S}$ be the subclass of $\mathcal{A}$ consisting of functions that are univalent in $\D$.

\medskip

Some of the subclasses of the class $\es$ that attract special attention are the class of starlike functions that maps the unit disk onto a starlike domain (each half line starting from the origin intersects the boundary of the domain,if existing, maximum once),
\[\es^*=\left\{f\in\A:\real \left[ \frac{zf'(z)}{f(z)}\right]>0, \, z\in\D \right\},\]
the class of convex functions (subclass of the class of starlike functions),
\[\mathcal{K}=\left\{f\in\A:\real \left[ 1+ \frac{zf''(z)}{f'(z)}\right]>0, \, z\in\D \right\},\]
and class
\[\U=\left\{f\in\A:\left| \left(\frac{z}{f(z)}\right)^2 f'(z) -1\right|<1, \, z\in\D \right\},\]
which is interesting because caries rare property that although $\es^*$ is very wide, $\U\nsubseteq\es^*$, nor vice versa.

\medskip

Although the famous Bieberbach conjecture $|a_n|\le n$ for $n\ge2,$ was proved by de Branges  in 1985 \cite{Bra85}, a great many other problems concerning the coefficients $a_n$ remain open and are studied over entire class $\es$, or, over its subclasses.

\medskip

One such problem is finding sharp estimates of logarithmic coefficient, $\gamma_n$, of a univalent function $f(z)=z+a_2z^2+a_3z^3+\cdots,$  defined by
  \be\label{log-co-1}
  F_f(z) := \log\frac{f(z)}{z}=2\sum_{n=1}^\infty \gamma_n z^n.
  \ee
From the relations \eqref{e1} and \eqref{log-co-1}, after  equating the coefficients we receive the next initial logarithmic coefficients:
\begin{equation}\label{eq-3}
\begin{split}
 \gamma_{1}&=\frac{a_{2}}{2},\\
 \gamma_{2}&=\frac{1}{2}\left(a_3-\frac{1}{2}a_{2}^{2}\right), \\
 \gamma_3 &=\frac{1}{2}\left(a_4-a_2a_3+\frac13a_2^3\right),\\
\gamma_4 &=\frac{1}{2}\left(a_5-a_2a_4-\frac{1}{2}a_3^2+a_{2}^{2}a_{3}-\frac{1}{4}a_{2}^{4}\right).
\end{split}
\end{equation}

Relatively little exact information is known about the coefficients. The natural conjecture $|\gamma_n|\le1/n$, inspired by the Koebe function (whose logarithmic coefficients are $1/n$) is false even in order of magnitude (see Duren \cite[Section 8.1]{duren}).
For the class $\es$ the sharp estimates of single logarithmic coefficients S are known only for $\gamma_1$ and $\gamma_2$, namely,
\[|\gamma_1|\le1\quad\mbox{and}\quad |\gamma_2|\le \frac12+\frac1e=0.635\ldots,\]
and are unknown for $n\ge3$. In this paper we give the estimates of the modulus of the third and the fourth logarithmic coefficients for the general class of univalent functions. 

\medskip

%
%

The upper bound of the Hankel determinant is a problem rediscovered and extensively studied in recent years. Over the class $\A$ of functions $f(z)=z+a_2z^2+a_3z^3+\cdots$ analytic on the unit disk, this determinant is defined by
\[
        H_{q,n}(f) = \left |
        \begin{array}{cccc}
        a_{n} & a_{n+1}& \ldots& a_{n+q-1}\\
        a_{n+1}&a_{n+2}& \ldots& a_{n+q}\\
        \vdots&\vdots&~&\vdots \\
        a_{n+q-1}& a_{n+q}&\ldots&a_{n+2q-2}\\
        \end{array}
        \right |,
\]
where $q\geq 1$ and $n\geq 1$. The second order Hankel determinants is
\[
 H_{2,2}(f) =  \left |
        \begin{array}{cc}
        a_2 & a_3\\
        a_3 & a_4\\
        \end{array}
        \right | = a_2a_4-a_{3}^2,
\]
and for the logarithmic coefficients:
\begin{equation}\label{loghank}
H_{2,1}(F_f/2) = \gamma_1\gamma_3-\gamma_2^2 = \frac14\left( a_2a_4-a_3^2 + \frac{1}{12}a_2^4\right) .
\end{equation}
One of the Hankel determinants of third order is 
\[ 
H_{3,1}(f) =  \left |
        \begin{array}{ccc}
        1 & a_2& a_3\\
        a_2 & a_3& a_4\\
        a_3 & a_4& a_5\\
        \end{array}
        \right | = a_3(a_2a_4-a_{3}^2)-a_4(a_4-a_2a_3)+a_5(a_3-a_2^2).
 \]

\medskip

For the general class $\es$ of univalent functions in the class $\A$ there are very few results concerning the Hankel determinant.  The best known for the second order case  is due to Hayman (\cite{hayman-68}), saying  that $|H_{2,n}(f)|\le An^{1/2}$, where $A$ is an absolute constant, and that this rate of growth is the best possible.
There are much more results for the subclasses of $\es$ and some references  are \cite{ janteng-06,janteng-07, DTV-book}. Much less is known about the bounds of the modulus of the Hankel determinant for the logarithmic coefficients. In \cite{Kowalczyk-22,Kowalczyk-23} the authors considered the cases of starlike, convex, strongly starlike and strongly convex functions and found the best possible results. In this paper we give estimate for the second order Hankel determinant for the general class of univalent functions.

\medskip

For the study of the problems defined above  we will use method based on Grunsky coefficients. This method is suitable for the class of univalent functions because, unlike its subclasses, there is no analytical characterisation for the functions in $\es$.

\medskip

Here are basic definitions and results  on Grunsky coefficients based on the book of N. A. Lebedev (\cite{Lebedev}).

\medskip

Let $f \in \mathcal{S}$ and let
\[
\log\frac{f(t)-f(z)}{t-z}=\sum_{p,q=0}^{\infty}\omega_{p,q}t^{p}z^{q},
\]
where $\omega_{p,q}$ are called Grunsky's coefficients with property $\omega_{p,q}=\omega_{q,p}$.
For those coefficients we have the next Grunsky's inequalities (\cite{duren,Lebedev}):
\be\label{eq 4}
\sum_{q=1}^{\infty}q \left|\sum_{p=1}^{\infty}\omega_{p,q}x_{p}\right|^{2}\leq \sum_{p=1}^{\infty}\frac{|x_{p}|^{2}}{p}
\ee
and
\be\label{eq 4-2}
\left|\sum_{p=1}^{\infty} \sum_{q=1}^{\infty} \omega_{p,q} x_{p} x_{q} \right|  \leq \sum_{p=1}^{\infty}\frac{|x_{p}|^{2}}{p},
\ee
where $x_{p}$ are arbitrary complex numbers such that $0< \sum_{p=1}^{\infty}\frac{|x_{p}|^{2}}{p}< +\infty$. If $\overline{\lim}_{p\to\infty} \sqrt[p]{|x_p|}<1$, then in \eqref{eq 4} we have equality if, and only if, the area of $\widehat{\mathbb{C}} \setminus f^{-1}(\D)$ is zero, where $f^{-1}(z)= \frac{1}{f(z)}$. In \eqref{eq 4-b} equality apears if, and only if, $\sum_{p=1}^{\infty}\omega_{p,q}x_{p} = \lambda \frac{1}{q} \overline{x}_q$, $\lambda=1,2,\ldots$, and $\lambda$ is some constant with $|\lambda|=1$.

\medskip

Further, it is well-known that if $f$ given by \eqref{e1}
belongs to $\mathcal{S}$, then also
\be\label{eq 5-mo-4}
f_{2}(z)=\sqrt{f(z^{2})}=z +c_{3}z^3+c_{5}z^{5}+\cdots
\ee
belongs to the class $\mathcal{S}$. Then for the function $f_{2}$ we have the appropriate Grunsky's
coefficients of the form $\omega_{2p-1,2q-1}$ and the inequalities \eqref{eq 4} and \eqref{eq 4-2} receive the forms:
\be\label{eq 6-mo-5}
\sum_{q=1}^{\infty}(2q-1) \left|\sum_{p=1}^{\infty}\omega_{2p-1,2q-1}x_{2p-1}\right|^{2}\leq \sum_{p=1}^{\infty}\frac{|x_{2p-1}|^{2}}{2p-1}
\ee
and
\be\label{eq 4-b}
\left|\sum_{p=1}^{\infty} \sum_{q=1}^{\infty} \omega_{2p-1,2q-1} x_{2p-1} x_{2q-1} \right|  \leq \sum_{p=1}^{\infty}\frac{|x_{2p-1}|^{2}}{2p-1},
\ee
respectively.

\medskip

Here, and further in the paper we omit the upper index (2) in  $\omega_{2p-1,2q-1}^{(2)}$ if compared with Lebedev's notation.

\medskip

From inequalities \eqref{eq 4-b} and \eqref{eq 6-mo-5}, when $x_{2p-1}=0$ and $p=2,3,\ldots$, we have
\begin{equation}\label{e7}
|\omega_{11} x_1 +\omega_{31} x_3 |^2 +3|\omega_{13} x_1 +\omega_{33} x_3 |^2 + 5|\omega_{15} x_1 +\omega_{35} x_3 |^2 \le |x_1|^2+\frac{|x_3|^2}{3}
\end{equation}
and
\begin{equation}\label{eq-7}
|\omega_{11} x_1^2 +2\omega_{13}x_1 x_3 +\omega_{33}x_3^2 |\le |x_1|^2+\frac{|x_3|^2}{3},
\end{equation}
respectively. From \eqref{e7}, for $x_1=1$ and $x_3=0$, we obtain
\begin{equation}\label{eq-8}
  |\omega_{11}|^2 + 3 |\omega_{13}|^2 + 5|\omega_{15}|^2 \le1,
\end{equation}
and for $x_1=0$ and $x_3=1$, we obtain
\begin{equation}\label{eq-9}
  |\omega_{13}|^2 + 3 |\omega_{33}|^2 + 5|\omega_{35}|^2 \le \frac13.
\end{equation}

\medskip

As it has been shown in \cite[p.57]{Lebedev}, if $f$ is given by \eqref{e1} then the coefficients $a_{2}$, $ a_{3}$, $ a_{4}$ and $a_5$ are expressed by Grunsky's coefficients  $\omega_{2p-1,2q-1}$ of the function $f_{2}$ given by
\eqref{eq 5-mo-4} in the following way:
\be\label{e9}
\begin{split}
a_{2}&=2\omega _{11},\\
a_{3}&=2\omega_{13}+3\omega_{11}^{2}, \\
a_{4}&=2\omega_{33}+8\omega_{11}\omega_{13}+\frac{10}{3}\omega_{11}^{3},\\
a_{5}&=2\omega_{35}+8\omega_{11}\omega_{33}+5\omega_{13}^{2}+18\omega_{11}^{2}\omega_{13}+\frac{7}{3}\omega_{11}^{4},\\
0&=3\omega_{15}-3\omega_{11}\omega_{13}+\omega_{11}^{3}-3\omega_{33},\\
0&=\omega_{17}-\omega_{35}-\omega_{11}\omega_{33}-\omega_{13}^{2}+\frac{1}{3}\omega_{11}^{4}.
\end{split}
\ee

\medskip

Further, due to the famous Koebe's 1/4 theorem, a function $f\in\es$ given by \eqref{e1} has its inverse at least on the disk with radius 1/4 with an expansion
\begin{equation}\label{eq-11}
f^{-1}(w) = w+A_2w^2+A_3w^3+\cdots.
\end{equation}
Using the identity $f(f^{-1}(w))=w$, and after comparing the appropriate coefficient, we receive
\begin{equation}\label{eq-12}
\begin{array}{l}
A_{2}=-a_{2}, \\
A _{3}=-a_{3}+2a_{2}^{2} , \\
A_{4}= -a_{4}+5a_{2}a_{3}-5a_{2}^{3},\\
A_{5}= -a_{5}+6a_{2}a_{4}-21a_{2}^{2}a_3+3a_3^2+14a_2^4.
\end{array}
\end{equation}

\medskip

In this paper we will make a survey of results obtained using the Grunsky coefficients, some of them sharp, dealing with estimates of the modulus of the third and the fourth logarithmic coefficients, of the modulus of the second and the third Hankel determinant and of the modulus of some coefficients of the inverse function, and some coefficient differences, for the general class of univalent functions where no analytical characterisation exists. More precisely, estimates are given

\medskip

\section{The third and the forth logarithmic coefficients}

Here is an estimate of the modulus of the third and the fourth logarithmic coefficient for the whole class of univalent functions. As an illustration of the application of the Grunsky coefficients, we give only the proof for the third logarithmic coefficient.

\begin{thm}\label{th1}
Let $f\in\mathcal{S}$ and be given by \eqref{e1}. Then
\begin{itemize}
  \item[($i$)] $|\gamma_{3}|\leq 0.5566178\ldots$ (\cite{124});
  \item[($ii$)] $|\gamma_4| \le 0.51059\ldots$ (\cite{106}).
\end{itemize}
\end{thm}

\begin{proof}
From \eqref{log-co-1}, after differentiation and comparation of coefficients we receive
\[\gamma_3 =\frac12\left(a_4-a_2a_3+\frac13a_2^3\right).\]
The fifth relation in \eqref{e9} gives
\be\label{e11n}\omega_{33} =  \omega_{15}-\omega_{11}\omega_{13}+\frac13\omega_{11}^3,  \ee
which, together with the other expressions from \eqref{e9} implies
\[\gamma_3=\omega_{33}+2\omega_{11}\omega_{13} = \omega_{15}+\omega_{11}\omega_{13}+\frac13\omega_{11}^3.\]
Therefore,
\be\label{e12n}
|\gamma_3| \le \frac13|\omega_{11}|^3+|\omega_{11}||\omega_{13}|+|\omega_{15}|.
\ee
Now, choosing $x_1=1$ and $x_3=0$ in \eqref{e7} we have
\[ |\omega_{11} |^2 +3|\omega_{13} |^2 + 5|\omega_{15}|^2 \le 1, \]
and also from here
\[ |\omega_{11} |^2 +3|\omega_{13} |^2  \le 1. \]
The last two relations imply
\be\label{e14n}
|\omega_{13}|\le\frac{1}{\sqrt{3}}\sqrt{1-|\omega_{11}|^2} \quad \mbox{and}\quad |\omega_{15}|\le\frac{1}{\sqrt{5}}\sqrt{1-|\omega_{11}|^2-3|\omega_{13}|^2}.
\ee
Using \eqref{e12n} and \eqref{e14n} we have
\[
|\gamma_3| \le \frac13|\omega_{11}|^3+|\omega_{11}||\omega_{13}|+\frac{1}{\sqrt{5}}\sqrt{1-|\omega_{11}|^2-3|\omega_{13}|^2}\equiv f_1(|\omega_{11}|,|\omega_{13}|),
\]
where
\[ f_1(x,y) = \frac13x^3+xy+\frac{1}{\sqrt{5}}\sqrt{1-x^2-3y^2}\]
and $0\le x\le1$, $0\le y \le \frac{1}{\sqrt3} \sqrt{1-x^2}$ ($|a_2|=|2\omega_{11}|\le2$ implies $0\le|\omega_{11}|\le1$).

\medskip

So, we need to find maximum of the function $f_1$ over the region $E_1=\Big\{ (x,y):0\le x\le1, 0\le y \le \frac{1}{\sqrt3} \sqrt{1-x^2} \Big\}$.

\medskip

The system
\[
\begin{cases}
\frac{\partial f_1}{\partial x} = x^2+y-\frac{x}{\sqrt5 \sqrt{1-x^2-3y^2}} = 0 \\[2mm]
\frac{\partial f_1}{\partial y} = x-\frac{3y}{\sqrt5 \sqrt{1-x^2-3y^2}} = 0
\end{cases},
\]
has only one solution in the interior of $E_1$, that is $(x_1,y_1)=(0.81267\ldots, 0.243532\ldots)$ such that $f_1(x_1,y_1) = 0.5566178\ldots$.

\medskip
Now, let consider the function $f_1$ on the boundary of $E_1$:
\begin{itemize}
  \item[-] $f_1(x,0)=\frac13x^3+\frac{1}{\sqrt5}\sqrt{1-x^2}\le \frac{1}{\sqrt5}=0.4472\ldots$ for $0\le x\le1$, with maximum obtained for $x=0$;
  \item[-] $f_1(0,y)=\frac{1}{\sqrt5}\sqrt{1-3y^2}\le \frac{1}{\sqrt5}=0.4472\ldots$ for $0\le y\le1/\sqrt3$, with maximum obtained for  $y=0$;
  \item[-] $f_1(1,y)=f_1(1,0)=\frac13$;
  \item[-] $f_1\left(x,\frac{1}{\sqrt3}\sqrt{1-x^2}\right) = \frac13x^3+\frac{1}{\sqrt3}x\sqrt{1-x^2}\le \frac{1}{\sqrt5}=0.4472\ldots$ for $0\le x\le1$, with maximum obtained for  $x=0.898344\ldots$.
\end{itemize}

\medskip

Summarizing the above analysis brings the conclusion that
\[|\gamma_3| \le f_1(x_1,y_1)=0.5566178\ldots.\]
\end{proof}

\medskip

Next is a result over the difference of the modulus of the third and the second, and the fourth and the third logarithmic coefficient. We give the proof or the first one since the application of the Grunsky coefficient method differs from the one in the previous theorem.

\bthm\label{th-1}[\cite{113}]
Let  $\gamma_{2}$, $\gamma_{3}$ and $\gamma_{4}$ be the logarithmic coefficients of function $f\in \mathcal{S}$. Then
$$|\gamma _{3}|-|\gamma_{2}|\leq \frac{1}{\sqrt{5}} \quad\text{and}\quad
|\gamma _{4}|-|\gamma_{3}|\leq \frac{1}{\sqrt{7}}.$$
\ethm

\begin{proof}
For our consideration we need the following facts.
From \eqref{eq 6-mo-5}, choosing $x_{2p-1}=0$ when $p=3,4,\ldots$, we have
\[
\begin{split}
&|\omega_{11}x_{1}+\omega_{31}x_{3}|^{2}+3|\omega_{13}x_{1}+\omega_{33}x_{3}|^{2}+
5|\omega_{15}x_{1}+\omega_{35}x_{3}|^{2}\\
+&7|\omega_{17} x_1 +\omega_{37} x_3 |^2\leq |x_1|^2+\frac{|x_3|^2}{3}.
\end{split}
\]
If additionally, $x_{1}=1$ and $x_{3}=0$, then
$$|\omega_{11}|^{2}+3|\omega_{13}|^{2}+5|\omega_{15}|^{2}+7|\omega_{17}|^{2}\leq 1,$$
and from here also
\[
\begin{split}
|\omega_{11}|^{2}+3|\omega_{13}|^{2}+5|\omega_{15}|^{2} &\leq 1,\\
|\omega_{11}|^{2}+3|\omega_{13}|^{2} &\leq 1,
\end{split}
\]
and $$|\omega_{11}|^{2}\leq1 .$$
From the previous inequalities we have
\begin{equation}\label{eq-9}
\begin{split}
|\omega_{11}|&\leq1,\\
|\omega_{13}|&\leq\frac{1}{\sqrt{3}}\sqrt{1-|\omega_{11}|^{2}}, \\
|\omega_{15}|&\leq \frac{1}{\sqrt{5}}\sqrt{1-|\omega_{11}|^{2}-3|\omega_{13}|^{2}},\\
|\omega_{17}|&\leq \frac{1}{\sqrt{7}}\sqrt{1-|\omega_{11}|^{2}-3|\omega_{13}|^{2}-5|\omega_{15}|^{2}}.
\end{split}
\end{equation}
Also, from the fifth relation in \eqref{e9} we obtain
\begin{equation}\label{eq-11}
\omega_{33}=\omega_{15}-\omega_{11}\omega_{13}+\frac{1}{3}\omega_{11}^{3}.
\end{equation}

\medskip

Now, for the first estimate of the theorem, using \eqref{eq-3} and \eqref{e9}, we have
\begin{equation}\label{eq-12}
\begin{split}
|\gamma _{3}|-|\gamma_{2}|&=\frac{1}{2}\left|a_4-a_2a_3+\frac13a_2^3\right|-
\frac{1}{2}\left|a_3-\frac{1}{2}a_{2}^{2}\right|\\
&=|\omega_{33}+2\omega_{11}\omega_{13}|-\left|\omega_{13}+\frac{1}{2}\omega_{11}^{2}\right|,
\end{split}
\end{equation}
and  after applying $|\omega_{11}|\leq1$, brings
\[
\begin{split}
|\gamma _{3}|-|\gamma_{2}|& \leq|\omega_{33}+2\omega_{11}\omega_{13}|-|\omega_{11}|\cdot \left|\omega_{13}+\frac{1}{2}\omega_{11}^{2}\right|\\
&\leq \left|(\omega_{33}+2\omega_{11}\omega_{13})-(\omega_{11})\left(\omega_{13}+\frac{1}{2}\omega_{11}^{2}\right)\right|\\
&=\left|\omega_{33}+\omega_{11}\omega_{13}-\frac{1}{2}\omega_{11}^{3}\right|.
\end{split}
\]
From here, having in mind the relation \eqref{eq-11} and inequalities \eqref{eq-9}, we receive
\begin{equation}\label{eq-14}
\begin{split}
|\gamma _{3}|-|\gamma_{2}|&\leq \left|\omega_{15}-\frac{1}{6}\omega_{11}^{3}\right| \leq |\omega_{15}|+\frac{1}{6}|\omega_{11}|^{3}\\
&\leq \frac{1}{\sqrt{5}}\sqrt{1-|\omega_{11}|^{2}-3|\omega_{13}|^{2}}+\frac{1}{6}|\omega_{11}|^{3}\\
&=:\Phi(|\omega_{11}|,|\omega_{13}|),
\end{split}
\end{equation}
where
\[\Phi(u,v)=\frac{1}{\sqrt{5}}\sqrt{1-u^{2}-3v^{2}}+\frac{1}{6}u^{3}\]
with domain
\[ \Omega = \left\{ (u,v) : 0\leq u \leq1,\, 0\leq v \leq \frac{1}{\sqrt{3}}\sqrt{1-u^{2}} \right\}.\]

\medskip

It remains to find the maximal value of $\Phi$ on the domain $\Omega$.

\medskip

Is easy to verify that  the function $\Phi(u,v)$  has no interior singular points in the domain $\Omega$ ($\Phi'_v(u,v)=0$, if, and only of, $v=0$).
On the edges:
\[
\begin{split}
\Phi(u,0)&=\frac{1}{\sqrt{5}}\sqrt{1-u^{2}}+\frac{1}{6}u^{3}\leq \Phi(0,0)= \frac{1}{\sqrt5}=0.44721\ldots ,\\
\Phi(0,v)&=\frac{1}{\sqrt{5}}\sqrt{1-3v^{2}}\leq\frac{1}{\sqrt{5}}=0.44721\ldots,\\
\Phi\left(u,\frac{1}{\sqrt{3}}\sqrt{1-u^{2}}\right)&=\frac{1}{6}u^3 \leq \frac{1}{6}=0.1(6).
\end{split}
\]

\medskip

Using all previous facts and \eqref{eq-14}, we finally conclude that
$$|\gamma _{3}|-|\gamma_{2}|\leq \frac{1}{\sqrt{5}}.$$
\end{proof}

Previous theorem leads to the conjecture for the difference of the moduli of two consecutive logarithmic coefficients of univalent functions.

\begin{conj}
If $\gamma_{n}$ is logarithmic coefficient of function $f\in \mathcal{S}$, then
$$|\gamma _{n}|-|\gamma_{n-1}|\leq \frac{1}{\sqrt{2n-1}},\quad n=3,4,\ldots .$$
\end{conj}

\medskip

\section{The second and the third Hankel determinants}

In this section we give estimates of the modulus of the second and the third Hankel determinants for the coefficients of $f$ and for the coefficients of the inverse of $f$. We give only the proof for the coefficients of the inverse  of $f$ since it is the only one that differs from the application of the Grunsky coefficients presented in the previous proofs.

\medskip

In \cite{OT-2021} only a range of the upper bounds of $|H_{2,2}(f)|$ and $|H_{3,1}(f)|$
for the class $\mathcal{S}$ were given. Namely,

\bthm\label{19-th 1}[\cite{OT-2021}] For the class $\mathcal{S}$ we have
\[
|H_{2,2}(f)|\leq A , \quad\mbox{where}\quad 1\leq A\leq \frac{11}{3}=3,66\ldots
\]
and
\[
|H_{3,1}(f)|\leq  B, \quad\mbox{where}\quad \frac49\leq B\leq \frac{32+\sqrt{285}}{15} = 3.258796\cdots
\]
\ethm

In \cite{120} the authors improved these results by proving:

\bthm\label{19-th 2}[\cite{120}]
For the class $\mathcal{S}$ we have the next estimations:
\begin{itemize}
  \item[$(i)$]  $|H_{2,2}(f)|\leq 1.3614\ldots ; $
  \item[$(ii)$]  $|H_{3,1}(f)|\leq 1.6787\ldots.  $
\end{itemize}
\ethm

\medskip

\medskip

The modulus of $H_{2,3}(f)$ was estimated In \cite{110} for the case of vanishing and non-vanishing second coefficient:

\bthm\label{22-th-2}[\cite{110}]
Let $f\in\mathcal{S}$ is given by \eqref{e1}. Then
\begin{itemize}
  \item[($i$)] $|H_{2,3}(f)|\le 2.02757\ldots$ if $a_2=0$;
  \item[($ii$)] $|H_{2,3}(f)|\le 4.8986977\ldots$ for every $f\in\es$.
\end{itemize}
\ethm

\medskip

In the next theorem we give the upper bound of the modulus of the second and the third Hankel determinant for the inverse functions of the functions from the class $\es$ (sharp for the second order determinant). The proof of the later is included since it differs from the previous ones.

\medskip

\begin{thm}[\cite{99}]\label{th-1}
Let $f\in\es$ and $f^{-1}$ be its inverse function. Then
\begin{itemize}
  \item[($i$)] $|H_{2,2}(f^{-1})| \le 3$ and the result is the best possible;
  \item[($ii$)] $|H_{3,1}(f^{-1})| \le  2.36639\ldots$.
\end{itemize}
\end{thm}

\begin{proof}
($i$) By the definition of $H_{2,2}$ and the relations \eqref{eq-12} we have
\[ H_{2,2}(f^{-1}) = A_2A_4-A_3^2 = a_2a_4-a_3^2-a_2^2(a_3-a_2^2), \]
and from here, by using \eqref{e9} and some calculations
\[  H_{2,2}(f^{-1}) = 4\omega_{11}\left(\omega_{33}-\omega_{11}\omega_{13}+\frac{5}{12}\omega_{11}^3\right) - 4\omega_{13}^2,\]
which implies
\begin{equation}\label{eq-14}
  |  H_{2,2}(f^{-1}) | \le 4|\omega_{11}| \left|\omega_{33}-\omega_{11}\omega_{13}+\frac{5}{12}\omega_{11}^3\right| + 4|\omega_{13}|^2.
\end{equation}

\medskip

If we choose $x_1=-\frac12\omega_{11}$ and $x_3=1$ in the relation \eqref{eq-7}, then we have
\begin{equation}\label{eq-13}
  \left|\frac14\omega_{11}^3-\omega_{11}\omega_{13}+\omega_{33}\right| \le \frac14 |\omega_{11}|^2+\frac13,
\end{equation}
i.e.,
\[ \left| \left(\omega_{33}-\omega_{11}\omega_{13}+\frac{5}{12}\omega_{11}^3\right) - \frac16\omega_{11}^3\right| \le \frac14 |\omega_{11}|^2+\frac13, \]
and
\[ \left| \omega_{33}-\omega_{11}\omega_{13}+\frac{5}{12}\omega_{11}^3\right| \le  \frac16|\omega_{11}|^3 + \frac14 |\omega_{11}|^2+\frac13.\]
From the last inequality and \eqref{eq-14} we get
\begin{equation}\label{eq-star}
   \left|H_{2,2}(f^{-1})\right| \le 4|\omega_{11}| \left[ \frac16|\omega_{11}|^3  + \frac14|\omega_{11}|^2 +\frac13\right] + 4|\omega_{13}|^2 .
\end{equation}
From \eqref{eq-8} we have $|\omega_{13}|^2 \le \frac13(1-|\omega_{11}|^2)$, and from \eqref{eq-star}:
\[
\begin{split}
\left|H_{2,2}(f^{-1})\right| \le \frac43 + \frac43|\omega_{11}| - \frac43|\omega_{11}|^2 + |\omega_{11}|^3 +\frac23|\omega_{11}|^4 = \psi(|\omega_{11}|),
\end{split}
\]
where $\psi(t)=\frac43+\frac43 t-\frac43t^2+t^3+\frac23t^4$ and $t=|\omega_{11}|\in[0,1]$ since by \eqref{e9}, for the class $\es$, $2|\omega_{11}| = |a_2|\le2$. Then, $$\psi'(t)= \frac43-\frac83t+3t^2+\frac83t^3 = \frac43(1-t)^2+\frac53t^2+\frac83t^3>0,$$
$t\in[0,1]$,  meaning that $\psi(t)\le\psi(1)=3$, $t\in[0,1]$. Therefore, $\left|H_{2,2}(f^{-1})\right|\le3$.

\medskip

The estimate is sharp due to the Koebe function $k(z)=\frac{z}{(1-z)^2}$ such that
\[
\begin{split}
 k^{-1}(w) &= \frac{2w+1-\sqrt{1+4w}}{2w} = w+\sum_{2}^\infty (-1)^{n+1}\frac{(2n)!}{(n+1)!n!}w^n \\
 &= w-2w^2+5w^3-14w^4+\cdots,
 \end{split}
 \]
and $H_{2,2}(k^{-1}) = A_2A_4-A_3^2=3.$
\end{proof}

\medskip

\begin{rem}
The estimate for the second Hankel determinant (Theorem \ref{th-1}(i) is of course valid, but even more sharp for all classes univalent functions containing the Koebe function, such as the classes of starlike, class $\U$, et c..
\end{rem}

\medskip

Now we will give sharp estimate of the modulus of the difference of the second and the third order Hankel determinants:

\medskip

\begin{thm}\label{cor-1}
 Let $f\in\es$. Then the following estimates are sharp:
 \begin{itemize}
   \item[($i$)] $ \left|  H_{2,2}(f^{-1})-H_{2,2}(f) \right| \le4$;
   \item[($ii$)] $ \left|  H_{3,1}(f^{-1})-H_{3,1}(f) \right| \le1$.
 \end{itemize}
\end{thm}

\begin{proof}$ $
   \begin{itemize}
   \item[($i$)] From \eqref{eq-13} we receive that for any function $f\in\es$,
   \[  \left|  H_{2,2}(f^{-1})-H_{2,2}(f) \right| = |a_2|^2 |a_3-a_2^2| \le4, \]
   since $|a_2|\le2$ and $|a_3-a_2^2|\le1$ (see \cite{duren}). The above estimate is sharp since for the Koebe function we have $H_{2,2}(k)=2\cdot4-3^2=-1$ and $H_{2,2}(k^{-1})=(-2)\cdot(-14)-5^2=3$, so that $ H_{2,2}(k^{-1})-H_{2,2}(k)  = 4$.

\medskip

   \item[($ii$)] 
   From the definition of the third Hankel determonant we have
\[ H_{3,1}(f^{-1}) = A_3(A_2A_4-A_{3}^2)-A_4(A_4-A_2A_3)+A_5(A_3-A_2^2).
\]
Using the relations \eqref{eq-12}, after some calculations, we receive
\[
\begin{split}
  &\quad H_{3,1}(f^{-1})\\
   &= a_3(a_2a_4-a_{3}^2)-a_4(a_4-a_2a_3)+a_5(a_3-a_2^2) - (a_3-a_2^2)^3\\
  &= H_{3,1}(f) - (a_3-a_2^2)^3,
\end{split}
\]
   and further, for $f\in\es$,
  \[  \left|  H_{3,1}(f^{-1})-H_{3,1}(f) \right| = |a_3-a_2^2|^3 \le1, \]
  for every $f\in\es$. The above inequality is also sharp as the function $f_1(z)=\frac{z}{1-z^2} = z+z^3+z^5+\cdots$ with $a_2=a_4=0$ and $a_3=a_5=1$ shows. Indeed, from \eqref{eq-12} we receive $A_2=A_4=0$, $A_3=-1$, $A_5=2$, i.e., $H_{3,1}(f_1)=0$, $H_{3,1}(f_1^{-1})=-1$ and $|H_{3,1}(f_1^{-1})-H_{3,1}(f_1)|=1$.
 \end{itemize}
\end{proof}

\medskip

\medskip

The following estimate (non-sharp) of the modulus of $H_{3,2}(f^{-1})$ was given in \cite{113}:

\bthm\label{22-th-2}[\cite{113}]
Let $f\in\mathcal{S}$ is given by \eqref{e1} and let $a_{2}=0$. Then
 $$|H_{3,2}(f^{-1})|\leq \frac{\sqrt{3}}{6\sqrt{7}}+2\sqrt{3}=3.57321\ldots.$$
\ethm

\medskip

Finally, in \cite{124} an estimate of the second Hankel dterminant for the logarithmic coefficients wa given:

\begin{thm}[\cite{124}]
Let $f\in\es$ is given by \eqref{e1}. Then
\[|H_{2,1}(F_f/2)| = \left| \gamma_1\gamma_3-\gamma_2^2 \right| \le \frac13.\]
\end{thm}

\medskip

\section{On some inequalities for coefficients}

In this section, we begin with estimates of the modulus of the coefficients of $f$ and of the second and the third Hankel determinant in the case of missing second or third coefficient.

\begin{thm}\label{th-1}
Let $f\in\es$ and be given by \eqref{e1} with $a_2=0$. Then
\begin{itemize}
\item[($i$)] $|a_3|\le1$ (\cite{OTT}),
\medskip
\item[($ii$)] $|a_4|\le\dfrac23=0.666\ldots$ (\cite{OTT}),
\medskip
\item[($iii$)]  $|a_5|\le \frac{3}{4}+\frac{1}{\sqrt{7}} = 1.12796\ldots$ (\cite{119}),
\medskip
\item[($iv$)] $|H_{2,2}(f)|\le1$ (\cite{OTT}),
\medskip
\item[($v$)] $|H_{3,1}(f)|\le 1.026\ldots$ (\cite{119}).
\end{itemize}
\end{thm}

\medskip

\begin{thm}[\cite{119}]\label{th-5}
Let $f\in\es$ and be  given by \eqref{e1}, with $a_3=0$. Then
\medskip
\begin{itemize}
\item[($i$)] $|a_2|\le1$,
\medskip
\item[($ii$)] $|a_4|\le \frac14\sqrt{\frac{21}{5}}+\frac{5}{8} = 1.1373\ldots $,
\medskip
\item[($iii$)]  $|a_5|\le 1.674896577\ldots$,
\medskip
\item[($iv$)] $|H_{2,2}(f)|\le 1.1373\ldots$,
\medskip
%
\end{itemize}
\end{thm}

\medskip

A long standing problem in the theory of univalent functions is to find  sharp upper and lower bounds for $|a_{n+1}|-|a_n|$, when $f\in\mathcal{S}$. Since  the Keobe function has coefficients $a_n=n$, it is natural to conjecture  that $||a_{n+1}|-|a_n||\le1$. As early as 1933, this was shown to be false even when $n=2,$ when Fekete and Szeg\"o \cite{FekSze33} obtained the sharp bounds
 $$ -1 \leq |a_3| - |a_2| \leq \frac{3}{4} + e^{-\lambda_0}(2e^{-\lambda_0}-1) = 1.029\ldots, $$
where $\lambda_0$ is the unique value of $\lambda$ in $0 < \lambda <1$, satisfying the equation $4\lambda = e^{\lambda}$.

\medskip

Hayman \cite{Hay63} showed that if $f \in {\mathcal S}$, then $| |a_{n+1}| - |a_n| | \leq C$, where $C$ is an absolute constant. The exact value of $C$ is unknown,  the best estimate to date
being $C=3.61\ldots$ \cite{Gri76}, which because of the sharp estimate above when $n=2$, cannot be reduced to $1$.

\medskip

\begin{thm}\label{th-3}
Let $f\in\es$ and be given by \eqref{e1}. Then
\begin{itemize}
  \item[($i$)] $|a_4|-|a_3| \le 1.75185\ldots.$ (\cite{119}),
  \item[($ii$)] $|a_5|-|a_3| \le \frac{2}{\sqrt7}=0.7559\ldots$ if $f$ is an odd function (\cite{119}),
  \item[($iii$)] $|a_2a_3-a_4| \le 2.10064\ldots$  (\cite{124}),
  \item[($iv$)] $|a_5|-|a_4| \le 2.3297\ldots$ (\cite{106}).
\end{itemize}
\end{thm}

\medskip

\begin{rem}
We believe that $|a_2a_3-a_4| \le 2$ is true for the class $\es$.
\end{rem}

\end{document}